\documentclass[11pt]{article}

\usepackage{amsmath,amssymb,amsfonts,amsthm}
\usepackage{amscd,amsbsy}
\usepackage{mathtools}
\usepackage[mathscr]{eucal}
\usepackage{booktabs,array}
\usepackage{graphicx,epstopdf,floatrow,subfig}
\usepackage{multirow,makecell}
\usepackage{enumerate}
\usepackage{bm}
\usepackage{xcolor}
\usepackage{threeparttable}
\usepackage{comment}
\usepackage{cite}
\usepackage{url}
\usepackage{hyperref}
\hypersetup{hidelinks}

\allowdisplaybreaks[4]

\newtheorem{theorem}{Theorem}
\newtheorem{proposition}[theorem]{Proposition}

\newtheorem{remark}[theorem]{Remark}

\title{\bfseries
	Some new criteria for the uniqueness\\
	of limit cycles in Rayleigh--Li\'enard systems
}

\author{
	Hebai Chen$^{1}$\thanks{Email: chen\_hebai@csu.edu.cn},
	Dongmei Xiao$^{2}$\thanks{Email: xiaodm@sjtu.edu.cn},
	Rui Zhang$^{1}$\thanks{Email: zhang\_rui@csu.edu.cn}
	\\[0.5em]
	{\footnotesize
		$^1$School of Mathematics and Statistics, HNP-LAMA,
		Central South University,
	}\\
	{\footnotesize
		Changsha, Hunan 410083, P. R. China
	}\\[0.3em]
	{\footnotesize
		$^2$School of Mathematical Sciences, CAM-Shanghai,
		Shanghai Jiao Tong University,
	}\\
	{\footnotesize
		Shanghai 200240, P. R. China
	}
}

\date{}

\begin{document}
	
	\maketitle
	\hrule
	
	\begin{abstract}
		In this paper, new uniqueness criteria for
		Rayleigh--Li\'enard systems with damping of positive
		odd degree in the velocity are established.
		Under suitable conditions, any limit cycle is unique
		and hyperbolic, with its stability determined by
		the damping signs.
		Neither symmetry nor monotonicity of the damping
		and restoring functions is required.
		Changes of variables and time yield further criteria
		for generalized Kukles systems.
		Applications to three polynomial families yield
		uniqueness and nonexistence results.
		In particular, we resolve an open uniqueness problem
		for a hybrid van der Pol--Rayleigh oscillator
		left open by Chen et al. (Phys. D, 2021).
		We also give an alternative uniqueness proof for
		a system with homogeneous components and obtain
		uniqueness and nonexistence results for a modified
		hybrid van der Pol--Rayleigh oscillator.
		
		\medskip
		\noindent
		\textbf{Keywords:}
		Limit cycle; uniqueness; second-order nonlinear
		differential equation; Rayleigh--Li\'enard system.
		
		\medskip
		\noindent
		\textbf{MSC (2020):} 34C07; 37C27.
	\end{abstract}
	
	\vspace{2em}
	\hrule
	\section{Introduction}
	The existence, number and stability of limit cycles are classical
	questions in the qualitative theory of planar differential equations.
	In particular, the counting problem is closely related to the second
	part of Hilbert's sixteenth problem. In 1923, Dulac \cite{Dulac1923}
	claimed that every fixed planar polynomial system has finitely many
	limit cycles, but a gap was later found in his proof. New proofs were
	subsequently proposed by Ilyashenko \cite{Ilyashenko1990,Ilyashenko1991} and \'Ecalle
	\cite{Ecalle1992}, and Yeung \cite{Yeung} subsequently identified a gap in the argument presented in Ilyashenko’s monograph \cite{Ilyashenko1991}. Recently, OpenAI
	\cite{OAI2026} released a preprint claiming uniform boundedness,
	whose arguments call for careful independent verification.
	Even establishing finiteness is challenging, let alone determining
	the exact number of limit cycles of a given system.
	
	These difficulties motivate the search for criteria establishing
	the existence, nonexistence and uniqueness of limit cycles, as well
	as the existence of exactly \(n\) limit cycles. A substantial part
	of the classical uniqueness theory concerns Li\'enard systems
	\[
	\dot x=y-F(x),\qquad \dot y=-g(x),
	\]
	and generalized Li\'enard systems
	\[
	\dot x=\phi(y)-F(x),\qquad \dot y=-g(x).
	\]
	Important contributions include those of Levinson and Smith
	\cite{LS}, Sansone \cite{Sansone1949}, Zeng \cite{Zeng},
	Liou and Cheng \cite{LC}, Wang and Kooij \cite{WK},
	Zeng, Zhang and Gao \cite{ZZG}, and Xiao and Zhang
	\cite{XZ,XZ2008}; see also \cite{Zh} and the references therein.
	Extending such criteria to oscillators with more general damping
	requires additional analysis.
	
	Uniqueness criteria also have practical significance. In nonlinear
	oscillator models, a stable limit cycle represents a self-sustained
	oscillation, and uniqueness helps determine whether different
	periodic oscillatory regimes can coexist. Second-order differential
	equations with nonlinear damping provide a natural framework for
	these questions. They include classical Li\'enard equations and
	hybrid van der Pol--Rayleigh oscillators arising in mechanical models
	and pedestrian force modelling; see \cite{NB,Zh,ETA,CTX2023}
	and the references therein. Such equations take the form
	\[
	\ddot x+f(x,\dot x)\dot x+g(x)=0.
	\]
	Setting \(y=\dot x\), we obtain the Rayleigh--Li\'enard system
	\[
	\dot x=y,\qquad \dot y=-g(x)-f(x,y)y.
	\]
	
	For systems with power-law restoring forces, Chen, Tang and Xiao
	\cite{CTX2023} developed a comparison method and obtained a
	uniqueness criterion based on a scaling inequality for the damping.
	Chen and Xiao \cite{CX} also established existence and uniqueness
	criteria for hybrid oscillators. These results apply to several
	smooth and nonsmooth van der Pol--Rayleigh-type models.
	
	In this paper, we first consider the Rayleigh--Li\'enard system
	\begin{equation}\label{eq:odd}
		\dot x=y,\qquad \dot y=-g(x)-f(x)y^m
	\end{equation}
	on the strip \((\alpha,\beta)\times\mathbb R\), where
	\(-\infty\leq\alpha<0<\beta\leq+\infty\) and \(m\) is a positive
	odd integer. Building on Sansone's method \cite{Sansone1949},
	we establish a uniqueness criterion for system~\eqref{eq:odd}.
	The criterion requires neither symmetry nor monotonicity of
	\(f\) or \(g\). It also establishes hyperbolicity and determines
	the stability of any limit cycle that exists, but does not itself
	assert existence.
	
	We then apply this criterion to generalized Kukles systems
	\cite{Cherkas1978,CLT},
	\[
	\dot x=y,\qquad
	\dot y=-g(x)-P_1(x)y-P_2(x)y^2-P_3(x)y^3.
	\]
	Under suitable conditions, we construct changes of variables and
	time that reduce these systems to the form of~\eqref{eq:odd} with
	\(m=1\) or \(m=3\). Applying the preceding criterion to the
	transformed systems yields further uniqueness results.
	Finally, we illustrate the criteria with three polynomial families:
	a hybrid van der Pol--Rayleigh oscillator \cite{CTX},
	a system with homogeneous components \cite{Gasull,ZL},
	and a modified hybrid van der Pol--Rayleigh oscillator
	\cite{CGLZ,LZ}.
	
	The paper is organized as follows. Section~2 states and proves
	the uniqueness criteria and develops the reductions of generalized
	Kukles systems. Section~3 presents their applications to the three
	polynomial families.

	\section{Main results}
	\subsection{A uniqueness criterion for limit cycles  of Rayleigh--Li\'enard systems}
	Consider the Rayleigh--Li\'enard system
	\eqref{eq:odd}
	on the strip $(\alpha,\beta)\times\mathbb R$, where
	$-\infty\leq\alpha<0<\beta\leq+\infty$ and $m$ is a positive
	odd integer. Set
	\[
	G(x)=\int_0^x g(s)\,ds.
	\]
	
	\begin{theorem}\label{thm:balanced}
		Assume that system~\eqref{eq:odd} satisfies the following conditions:
		\begin{enumerate}
			\item[(i)]
			$f,g\in C^1((\alpha,\beta))$;
			
			\item[(ii)]
			$xg(x)>0$ for every $x\in(\alpha,\beta)\setminus\{0\}$;
			
			\item[(iii)]
			there exist $\alpha<c_1<0<c_2<\beta$ with
			$G(c_1)=G(c_2)$ such that
			$f(x)<0$ (resp. $f(x)>0$) for
			$x\in(c_1,c_2)\setminus\{0\}$, and
			$f(x)>0$ (resp. $f(x)<0$) for
			$x\in(\alpha,c_1)\cup(c_2,\beta)$.
		\end{enumerate}
		Then system~\eqref{eq:odd} has at most one limit cycle
		contained in the strip $\alpha<x<\beta$.
		If such a limit cycle exists, it is hyperbolic and
		stable (resp. unstable).
	\end{theorem}
	
	\begin{remark}
		Notice that neither symmetry nor monotonicity of $f$ and $g$
		is required in Theorem~\ref{thm:balanced}, and the case $f(0)=0$
		is also allowed. The theorem gives the uniqueness and
		hyperbolicity of a limit cycle, provided that it exists,
		but does not guarantee its existence.
		
		Theorem~\ref{thm:balanced} may be viewed as an extension of
		Sansone's uniqueness theorem
		(see \cite{Sansone1949} or \cite[Theorem~4.3]{Zh})
		to systems with the positive odd-power damping term $-f(x)y^m$
		and a general restoring force $g(x)$.
		When $m=1$, system~\eqref{eq:odd} is a Li\'enard system.
		If $g(x)=x$, the condition $G(c_1)=G(c_2)$
		implies $c_1=-c_2$, and  Theorem \ref{thm:balanced} recovers the uniqueness
		and stability conclusions of Sansone's theorem.
		For general $g$, this condition requires the two sign-change
		points to lie at the same potential level, but does not
		require symmetry with respect to the origin.
	\end{remark}
	
	\subsection{Reduction of generalized Kukles systems}
	
	The preceding criterion will now be applied to generalized Kukles systems through changes of variables and time.
	
	Consider the following 
	planar Rayleigh--Li\'enard system
	\begin{equation}\label{k}
		\begin{aligned}
			\frac{dx}{dt}&=y,		
			\\
			\frac{dy}{dt}&=-g(x)-P_1(x)y-P_2(x)y^2-P_3(x)y^3,
		\end{aligned}
	\end{equation}
	on the strip $(\alpha,\beta)\times\mathbb R$, where
	$-\infty\leq\alpha<0<\beta\leq+\infty$. 
	This system is known as a generalized Kukles system
	\cite{Cherkas1978,CLT}, as it extends the classical cubic
	system studied by Kukles \cite{kukles} while retaining
	the cubic dependence on $y$.
	
	Assume that $g(x), P_1(x), P_2(x), P_3(x)\in C^1((\alpha,\beta))$  and $xg(x)>0$ for $x\in (\alpha, \beta)\setminus\{0\}$.  Let $\varphi\in C^2((\alpha,\beta))$. On the region
	\[
	U_\varphi
	=\{(x,y)\in(\alpha,\beta)\times\mathbb R:
	1+\varphi(x)y>0\},
	\]
	consider the change of variables and time
	\begin{equation}\label{4}
		X=x,\qquad
		Y=\frac{y}{1+\varphi(x)y},\qquad
		d\tau=(1+\varphi(x)y)\,dt.
	\end{equation}
	A direct calculation shows that transformation~\eqref{4}
	takes system~\eqref{k} into
	\begin{equation}\label{7}
		\begin{aligned}
			\frac{dX}{d\tau}&=Y,\\
			\frac{dY}{d\tau}
			&=-g(X)-A_\varphi(X)Y
			-B_\varphi(X)Y^2-C_\varphi(X)Y^3,
		\end{aligned}
	\end{equation}
	where
	\begin{equation}\label{ABC}
		\begin{aligned}
			A_\varphi(X)
			&=P_1(X)-3g(X)\varphi(X),\\
			B_\varphi(X)
			&=P_2(X)-2P_1(X)\varphi(X)
			+3g(X)\varphi^2(X),\\
			C_\varphi(X)
			&=P_3(X)+\varphi'(X)-P_2(X)\varphi(X)
			+P_1(X)\varphi^2(X)-g(X)\varphi^3(X).
		\end{aligned}
	\end{equation}
	Note that $\varphi\in C^2((\alpha,\beta))$ ensures that
	$A_\varphi,B_\varphi,C_\varphi\in C^1((\alpha,\beta))$
	in system~\eqref{7}.
	Moreover, the coordinate change in~\eqref{4}
	maps $U_\varphi$ diffeomorphically onto
	\[
	V_\varphi
	=\{(X,Y)\in(\alpha,\beta)\times\mathbb R:
	1-\varphi(X)Y>0\}.
	\]
	Indeed, its Jacobian determinant is
	\[
	\det\left(\frac{\partial(X,Y)}{\partial(x,y)}\right)
	=
	\det \begin{pmatrix}
		1 & 0\\[2mm]
		-\dfrac{\varphi'(x)y^2}{(1+\varphi(x)y)^2}
		& \dfrac{1}{(1+\varphi(x)y)^2} 
	\end{pmatrix}=\dfrac{1}{(1+\varphi(x)y)^2}>0,
	\]
	and its inverse on $V_\varphi$ is given by
	\[
	x=X,\qquad
	y=\frac{Y}{1-\varphi(X)Y}.
	\]
	Since the time reparametrization is positive on
	$U_\varphi$, the limit cycles of system~\eqref{k}
	contained in $U_\varphi$ correspond bijectively
	to those of system~\eqref{7} contained in
	$V_\varphi$, with stability and hyperbolicity preserved.
	
	To make the quadratic term in system~\eqref{7} vanish, we  set
	\[
	\rho(X)=\exp\left(\int_0^X B_\varphi(s)\,ds\right)
	\]
	and consider the further transformations
	\[
	\xi=X,\qquad
	\eta=\rho(X)Y,\qquad
	d\sigma=\frac{1}{\rho(X)}d\tau.
	\]
	Since $\rho'=B_\varphi\rho$, system~\eqref{7}
	becomes
	\begin{equation}\label{8}
		\begin{aligned}
			\frac{d\xi}{d\sigma}&=\eta,\\
			\frac{d\eta}{d\sigma}
			&=-\rho^2(\xi)g(\xi)
			-\rho(\xi)A_\varphi(\xi)\eta
			-\frac{C_\varphi(\xi)}{\rho(\xi)}\eta^3.
		\end{aligned}
	\end{equation}
	This  change is defined on the whole strip
	$(\alpha,\beta)\times\mathbb R$ and maps $V_\varphi$
	diffeomorphically onto
	\[
	W_\varphi
	=\left\{(\xi,\eta)\in(\alpha,\beta)\times\mathbb R:
	1-\frac{\varphi(\xi)}{\rho(\xi)}\eta>0\right\}.
	\]
	Indeed, its Jacobian determinant is
	\[
	\det\left(\frac{\partial(\xi,\eta)}{\partial(X,Y)}\right)
	=
	\det
	\begin{pmatrix}
		1 & 0\\[2mm]
		B_\varphi(X)\rho(X)Y & \rho(X)
	\end{pmatrix}
	=\rho(X)>0,
	\]
	and its inverse is
	\[
	X=\xi,\qquad Y=\frac{\eta}{\rho(\xi)}.
	\]
	The time reparametrization is also positive.
	Consequently, the two transformations establish
	a bijective correspondence between the limit cycles
	of system~\eqref{k} contained in $U_\varphi$ and
	those of system~\eqref{8} contained in $W_\varphi$,
	preserving stability and hyperbolicity.

	If $\varphi$ can be chosen so that $A_\varphi\equiv0$
	or $C_\varphi\equiv0$, then system~\eqref{8} reduces
	to the form~\eqref{eq:odd}, with $m=3$ or $m=1$,
	respectively. Applying Theorem~\ref{thm:balanced}
	to the resulting system yields the following criteria.
	\subsubsection{Eliminating the cubic term in system \eqref{8}}
	\begin{theorem}\label{thm2}
		Assume that system~\eqref{k} satisfies the following
		conditions:
		\begin{enumerate}
			\item[(i)]
			$g,P_1,P_2,P_3\in C^1((\alpha,\beta))$;
			
			\item[(ii)]
			$xg(x)>0$ for every
			$x\in(\alpha,\beta)\setminus\{0\}$;
			
			\item[(iii)]
			there exists $\varphi_0\in\mathbb R$ such that
			the Cauchy initial value problem
			\begin{equation}\label{22}
				\begin{aligned}
					&\frac{d\varphi}{dx}
					=-P_3(x)+P_2(x)\varphi(x)
					-P_1(x)\varphi^2(x)+g(x)\varphi^3(x),\\
					&\varphi(0)=\varphi_0,
				\end{aligned}
			\end{equation}
			admits a solution $\varphi(x)$ defined on the whole
			interval $(\alpha,\beta)$;
			
			\item[(iv)]
			there exist $\alpha<c_1<0<c_2<\beta$ with
			$\widehat G_\varphi(c_1)=\widehat G_\varphi(c_2)$
			such that
			$A_\varphi(x)<0$ (resp. $A_\varphi(x)>0$)
			for $x\in(c_1,c_2)\setminus\{0\}$, and
			$A_\varphi(x)>0$ (resp. $A_\varphi(x)<0$)
			for $x\in(\alpha,c_1)\cup(c_2,\beta)$.
			Here the functions $A_\varphi$ and $B_\varphi$
			defined in~\eqref{ABC} are
			\[
			\begin{aligned}
				A_\varphi(x)&=P_1(x)-3g(x)\varphi(x),\\
				B_\varphi(x)&=P_2(x)-2P_1(x)\varphi(x)
				+3g(x)\varphi^2(x),
			\end{aligned}
			\]
			and
			\[
			\widehat G_\varphi(x)
			=\int_0^x \left(
			g(s)\exp\left(2\int_0^s B_\varphi(r)\,dr\right)\right)\,ds.
			\]
		\end{enumerate}
		Then system~\eqref{k} has at most one limit cycle
		contained in the strip $\alpha<x<\beta$.
		If such a limit cycle exists, it is hyperbolic and
		stable (resp. unstable).
	\end{theorem}
	\begin{remark}
		Since the right-hand side of~\eqref{22} is $C^1$
		in $(x,\varphi)$, the Cauchy initial value problem
		admits a unique local solution for every
		$\varphi_0\in\mathbb R$.
		However, this solution need not extend to the whole
		interval $(\alpha,\beta)$, so global existence is
		explicitly assumed in condition~(iii).
		Furthermore, equation~\eqref{22} is an Abel equation, for which an explicit solution
		is generally difficult to obtain.
	\end{remark}
	\subsubsection{Eliminating the linear term in system \eqref{8}}
	
	We next consider the alternative reduction obtained by choosing $\varphi$ so that $A_\varphi\equiv0$.
	
	\begin{theorem}\label{thm3}
		Assume that system~\eqref{k} satisfies the following
		conditions:
		\begin{enumerate}
			\item[(i)]
			$g,P_1,P_2,P_3\in C^1((\alpha,\beta))$;
			
			\item[(ii)]
			$xg(x)>0$ for every
			$x\in(\alpha,\beta)\setminus\{0\}$;
			
			\item[(iii)]
			the function
			\[
			\varphi(x)=\frac{P_1(x)}{3g(x)},
			\qquad x\in(\alpha,\beta)\setminus\{0\},
			\]
			extends to a function
			$\varphi\in C^2((\alpha,\beta))$;
			
			\item[(iv)]
			there exist $\alpha<c_1<0<c_2<\beta$ with
			$\widehat G_\varphi(c_1)=\widehat G_\varphi(c_2)$
			such that
			$C_\varphi(x)<0$ (resp. $C_\varphi(x)>0$)
			for $x\in(c_1,c_2)\setminus\{0\}$, and
			$C_\varphi(x)>0$ (resp. $C_\varphi(x)<0$)
			for $x\in(\alpha,c_1)\cup(c_2,\beta)$.
			Here the functions $B_\varphi$ and $C_\varphi$
			defined in~\eqref{ABC} are
			\[
			\begin{aligned}
				B_\varphi(x)
				&=P_2(x)-2P_1(x)\varphi(x)
				+3g(x)\varphi^2(x),\\
				C_\varphi(x)
				&=P_3(x)+\varphi'(x)-P_2(x)\varphi(x)
				+P_1(x)\varphi^2(x)-g(x)\varphi^3(x),
			\end{aligned}
			\]
			and
			\[
			\widehat G_\varphi(x)
			=\int_0^x
			g(s)\exp\left(2\int_0^s B_\varphi(r)\,dr\right)\,ds.
			\]
		\end{enumerate}
		Then system~\eqref{k} has at most one limit cycle
		contained in
		\[
		U_\varphi
		=\{(x,y)\in(\alpha,\beta)\times\mathbb R:
		1+\varphi(x)y>0\}.
		\]
		If such a limit cycle exists, it is hyperbolic and
		stable (resp. unstable).
	\end{theorem}
	\begin{remark}
		Although $\varphi(x)=P_1(x)/(3g(x))$ is explicitly
		given for $x\ne0$ in condition~(iii), conditions~(i)
		and~(ii) do not guarantee that it extends to a $C^2$
		function at $x=0$, where $g(0)=0$.
		For instance, if $g(x)=x^3/3$ and $P_1(x)=x^2$,
		then $\varphi(x)=1/x$, which does not even admit
		a continuous extension at $x=0$.
		Thus condition~(iii) must be checked separately;
		if it fails, Theorem~\ref{thm3} cannot be applied.
		Furthermore, the uniqueness conclusion is restricted
		to limit cycles contained in $U_\varphi$, where
		transformation~\eqref{4} is a diffeomorphism onto
		$V_\varphi$ and preserves time orientation.
		
	\end{remark}
	\subsection{Proofs of the main results}
	
	\begin{proof}[Proof of Theorem~\ref{thm:balanced}]
		Here we only consider the conditions outside the parentheses.
		The conditions in parentheses can be treated by the transformation
		$(x,y,t)\mapsto(x,-y,-t)$.

		It follows from conditions~(i) and~(ii) that $g(0)=0$
		and the origin is the unique equilibrium of
		system~\eqref{eq:odd}. Thus every limit cycle surrounds
		the origin. Set
		\[
		h=G(c_1)=G(c_2),\qquad
		\mathcal E(x,y)=\frac{y^2}{2}+G(x).
		\]
		Along solutions of system~\eqref{eq:odd}, we have
		\begin{equation}\notag
			\dot{\mathcal E}=-f(x)y^{m+1}.
		\end{equation}
		Since $xg(x)>0$ for $x\ne0$,
		\[
		G(x)<h \quad \text{for }x\in(c_1,c_2),
		\qquad
		G(x)>h \quad \text{for }
		x\in(\alpha,c_1)\cup(c_2,\beta).
		\]
		
		Let $\Gamma$ be a closed orbit of system~\eqref{eq:odd} in the strip $\alpha<x<\beta$.
		We claim that $\Gamma\subset\{(x,y)\in (\alpha,\beta)\times \mathbb R: \mathcal E(x,y)>h\}$.
		The level curve $\mathcal E=h$ is a simple closed curve
		surrounding the origin. On this curve,
		\[
		\dot{\mathcal E}=-f(x)y^{m+1}\geq0,
		\]
		with strict inequality except possibly at $(c_1,0)$,
		$(c_2,0)$ and the two intersections with the $y$-axis.
		At every point where the inequality is strict,
		the orbit through that point crosses the level
		curve from $\{\mathcal E<h\}$ to $\{\mathcal E>h\}$. 
		
		The same crossing direction holds at the exceptional
		points. To determine the crossing direction at $(c_1,0)$,
		compare system~\eqref{eq:odd} with the conservative system
		\begin{equation}\label{1}
			\dot x=y,\qquad \dot y=-g(x).
		\end{equation}
		Let $\gamma_1$ and $\gamma_2$ be short orbit arcs of
		systems~\eqref{eq:odd} and~\eqref{1},
		respectively, through $(c_1,0)$. Notice that $\gamma_2$
		coincides locally with the level curve $\mathcal E=h$.
		Since $g(c_1)<0$, both arcs can locally be written as
		$x=X_1(y)$ and $x=X_2(y)$ for $|y|<\varepsilon$, where
		$X_1(0)=X_2(0)=c_1$ and
		\[
		X_1'(y)=
		\frac{y}{-g(X_1(y))-f(X_1(y))y^m},
		\qquad
		X_2'(y)=\frac{y}{-g(X_2(y))}.
		\]
		Both denominators are positive in a sufficiently small
		neighborhood of $(c_1,0)$. Moreover, $X_j'(y)$ has the
		sign of $y$, so $X_j(y)>c_1$ for $y\ne0$ sufficiently
		small. For $c_1<x<c_2$ and $y\ne0$ in this neighborhood,
		\[
		\frac{y}{-g(x)-f(x)y^m}
		-\frac{y}{-g(x)}
		=
		\frac{f(x)y^{m+1}}
		{-g(x)\left(-g(x)-f(x)y^m\right)}<0.
		\]
		Thus the strict comparison theorem (see  \cite[Theorem 6.1 of Chapter 1]{Hale})
		gives $
		X_1(y)<X_2(y)$  for  $0<y<\varepsilon$ and $X_1(y)>X_2(y)$ for $-\varepsilon<y<0$. 
		Hence $\gamma_1$ crosses $\{\mathcal E=h\}$
		from inside to outside at $(c_1,0)$.
		The same conclusions at $(c_2,0)$ and at either intersection with the $y$-axis follow by an analogous	comparison. Consequently, $\Gamma$ cannot intersect $\{\mathcal E=h\}$. Next, we show that
		$\Gamma$ cannot  lie entirely in $\{\mathcal E<h\}$.
		Indeed, this region is contained in the strip
		$c_1<x<c_2$, where
		\[
		\dot{\mathcal E}=-f(x)y^{m+1}\geq0,
		\]
		with strict inequality whenever $x\ne0$ and $y\ne0$.
		Thus, if $\Gamma\subset\{\mathcal E<h\}$, then
		$\dot{\mathcal E}$ is nonnegative along $\Gamma$
		and strictly positive on a nonempty arc.
		Integration over one period gives
		\[
		0=\oint_\Gamma \dot{\mathcal E}\,dt>0,
		\]
		a contradiction. Therefore,
		\[
		\Gamma\subset
		\{(x,y)\in(\alpha,\beta)\times\mathbb R:
		\mathcal E(x,y)>h\}.
		\] 
		
		We now show that the integral of the divergence
		$-mf(x)y^{m-1}$ of the vector field defined by
		system~\eqref{eq:odd} along any nonconstant closed 
		orbit $\Gamma$ is negative, i.e.,
		\begin{equation}\label{2}
			-m\oint_\Gamma f(x)y^{m-1}\,dt<0.
		\end{equation}
		This implies that every nonconstant closed  orbit
		is a hyperbolic stable limit cycle. Suppose that two limit cycles existed. They would be nested, and the compact annulus
		between them would contain no equilibria. The boundary separating their basins of
		attraction would contain a nonempty compact invariant set disjoint from both cycles.
		By the Poincar\'e–Bendixson theorem, this set would contain a periodic orbit that is
		not attracting, contradicting \eqref{2}. Hence the limit cycle is unique.

		It remains to prove \eqref{2}.
		Set
		\[
		D(x,y)=y^2+2G(x)-2h.
		\]
		Since $D>0$ on $\Gamma$ and
		$\dot D=-2f(x)y^{m+1}$, periodicity gives
		\begin{equation}\label{3}
			0=\oint_\Gamma d(\ln (D))=\oint_\Gamma \frac{\dot D}{D}\,dt
			=-2\oint_\Gamma \frac{f(x)y^{m+1}}{D}\,dt.
		\end{equation}
		On the parts of $\Gamma$ where $x<c_1$ or $x>c_2$,
		we have $G(x)>h$ and $f(x)>0$. Hence
		\[
		\frac{f(x)y^{m+1}}{D}
		=
		\frac{y^2}{y^2+2(G(x)-h)}f(x)y^{m-1}
		\leq f(x)y^{m-1}.
		\]
		On the parts where $x\in(c_1,c_2)\setminus\{0\}$,
		we have $G(x)<h$, $f(x)<0$ and $y\ne0$, so
		\[
		\frac{f(x)y^{m+1}}{D}
		=
		\frac{y^2}{y^2-2(h-G(x))}f(x)y^{m-1}
		<f(x)y^{m-1}.
		\]
		The comparison is also non-strictly valid at
		$x=0,c_1,c_2$, by continuity.
		Since $\Gamma$ surrounds the origin, the strict
		inequality holds on a nonempty arc. Therefore,
		by \eqref{3},
		\[
		\oint_\Gamma f(x)y^{m-1}\,dt
		>
		\oint_\Gamma\frac{f(x)y^{m+1}}{D}\,dt
		=0.
		\]
		This proves \eqref{2} and
		completes the proof. 
	\end{proof}
	
	\begin{proof}[Proof of Theorem~\ref{thm2}]
		
		We first show that every  closed orbit of
		system~\eqref{k} in the strip $\alpha<x<\beta$
		is contained in $U_\varphi$.
		By condition~(i) and the ordinary differential equation~\eqref{22}
		in condition~(iii), we have
		$\varphi\in C^2((\alpha,\beta))$ and $C_\varphi\equiv0$.
		Set
		\[
		L(x,y)=1+\varphi(x)y.
		\]
		Along solutions of system~\eqref{k}, we obtain
		\[
		\frac{dL}{dt}
		=L\left(
		-g(x)\varphi(x)
		+\bigl(g(x)\varphi^2(x)-P_1(x)\varphi(x)\bigr)y
		-P_3(x)y^2
		\right).
		\]
		Since $\varphi(x)\ne0$ whenever $L(x,y)=0$,
		the set $\{L=0\}$ is a smooth invariant curve.
		By uniqueness of solutions, every closed orbit
		either lies entirely on this curve or is contained
		in one of the regions $\{L>0\}$ and $\{L<0\}$.
		On the other hand, every closed orbit 
		contains a point $(x_0,0)$ on the $x$-axis, where
		$L(x_0,0)=1>0$.
		Therefore, $L>0$ along the entire periodic orbit,
		which proves its containment in $U_\varphi$.
		
		Since $C_\varphi\equiv0$, system~\eqref{8} reduces to
		\[
		\frac{d\xi}{d\sigma}=\eta,\qquad
		\frac{d\eta}{d\sigma}
		=-\widehat g(\xi)-\widehat f(\xi)\eta,
		\]
		where
		\[
		\widehat g(\xi)=\rho^2(\xi)g(\xi),\qquad
		\widehat f(\xi)=\rho(\xi)A_\varphi(\xi).
		\]
		These functions belong to $C^1((\alpha,\beta))$.
		Since $\rho>0$, $\widehat f$ has the same sign
		as $A_\varphi$, and $\xi\widehat g(\xi)>0$
		for $\xi\ne0$. Moreover,
		\[
		\int_0^\xi \widehat g(s)\,ds
		=\widehat G_\varphi(\xi).
		\]
		Consequently, condition~(iv) verifies condition~(iii)
		of Theorem~\ref{thm:balanced}.
		Thus system~\eqref{8} has at most one limit cycle
		in the strip $\alpha<\xi<\beta$, and any such cycle
		is hyperbolic and stable (resp. unstable).
		Every limit cycle of system~\eqref{k}  lies in $U_\varphi$ and corresponds to a limit
		cycle of system~\eqref{8} in $W_\varphi$,
		with stability and hyperbolicity preserved.
		Therefore, system~\eqref{k} has at most one limit cycle
		in the strip $\alpha<x<\beta$.
		If such a limit cycle exists, it is hyperbolic and
		stable (resp. unstable).
	\end{proof}
	
	\begin{proof}[Proof of Theorem~\ref{thm3}]
		By condition~(iii), we have
		$P_1(x)=3g(x)\varphi(x)$ for $x\ne0$.
		By continuity, this identity also holds at $x=0$.
		Hence $A_\varphi\equiv0$, and system~\eqref{8}
		reduces to
		\[
		\frac{d\xi}{d\sigma}=\eta,\qquad
		\frac{d\eta}{d\sigma}
		=-\widehat g(\xi)-\widehat f(\xi)\eta^3,
		\]
		where
		\[
		\widehat g(\xi)=\rho^2(\xi)g(\xi),\qquad
		\widehat f(\xi)=\frac{C_\varphi(\xi)}{\rho(\xi)}.
		\]
		By conditions~(i) and~(iii), these functions belong
		to $C^1((\alpha,\beta))$.
		Since $\rho>0$, $\widehat f$ has the same sign
		as $C_\varphi$, and $\xi\widehat g(\xi)>0$
		for $\xi\ne0$. Moreover,
		\[
		\int_0^\xi \widehat g(s)\,ds
		=\widehat G_\varphi(\xi).
		\]
		Consequently, condition~(iv) verifies condition~(iii)
		of Theorem~\ref{thm:balanced}.
		Thus, by that theorem with $m=3$, system~\eqref{8}
		has at most one limit cycle in the strip
		$\alpha<\xi<\beta$, and any such cycle is
		hyperbolic and stable (resp. unstable).
		
		By the one-to-one correspondence between the limit
		cycles of system~\eqref{k} in $U_\varphi$ and those
		of system~\eqref{8} in $W_\varphi$, system~\eqref{k}
		has at most one limit cycle contained in $U_\varphi$.
		If such a limit cycle exists, it is hyperbolic and
		stable (resp. unstable).
		If every limit cycle in the original strip is
		contained in $U_\varphi$, the same conclusions
		hold throughout the strip.
	\end{proof}
	\section{Applications}\label{sec:applications}We apply the preceding criteria to three polynomial systems.
	In the first example, the linear term can be eliminated
	explicitly. In the other two examples, we eliminate
	the cubic term using solutions of the Abel equation~\eqref{22}.
	The first application resolves an open uniqueness problem,
	the second provides an alternative proof of a known
	uniqueness result, and the third establishes uniqueness
	in part of the parameter space and gives sufficient
	conditions for the nonexistence of periodic orbits.

	\subsection{A hybrid van der Pol--Rayleigh oscillator}
	A hybrid van der Pol--Rayleigh oscillator 
	\begin{equation}\label{CTX}
		\dot x=y,\qquad \dot y=-x-\alpha y-\beta y^3- x^2y,
	\end{equation}
	has been used to model
	the self-sustained walking behaviors of humans or bipedal robots in the literature; see  \cite{FD}.
	Chen, Tang and Xiao  \cite{CTX} studied the
	global dynamics of  system \eqref{CTX}.
	They noted that
	uniqueness of limit cycles remained open when 
	$\alpha=0$ and $-1<\beta<-1/3$; see \cite[p.4]{CTX}.
	Although existence was established for $\beta<-1/3$
	sufficiently close to $-1/3$, uniqueness in this parameter
	range was adopted as an assumption in their subsequent
	analysis.
	This unresolved problem motivates our first application,
	which concerns
	\begin{equation}\label{12}
		\dot x=y,\qquad
		\dot y=-x-x^2y-\beta y^3,
		\qquad \beta\in\mathbb R.
	\end{equation}
	
	\begin{proposition}\label{co1}
		If $-1<\beta<-1/3$, system~\eqref{12}
		has at most one limit cycle in $\mathbb R^2$.
		If such a limit cycle exists, it is hyperbolic and stable.
		If $\beta\notin(-1,-1/3)$, the system has no
		closed orbit.
	\end{proposition}
	\begin{proof}
		
		First, we claim that every closed orbit of
		system~\eqref{12} lies in
		\[
		U_0=\{(x,y)\in\mathbb R^2:1+xy>0\}.
		\]
		Indeed, on the curve $1+xy=0$, the slope of the
		vector field satisfies
		\[
		\left.\frac{dy}{dx}\right|_{xy=-1}
		=-\frac{x+x^2y+\beta y^3}{y}
		=-\beta y^2.
		\]
		Since the slope of the curve $1+xy=0$ is $y^2$,
		we obtain
		\[
		\left.\frac{dy}{dx}\right|_{xy=-1}
		\begin{cases}
			>y^2, & \text{if }\beta<-1,\\[4pt]
			<y^2, & \text{if }\beta>-1.
		\end{cases}
		\]
		That is, every intersection between an orbit and
		$\{xy=-1\}$ has the same crossing direction, which
		precludes a closed orbit from intersecting this curve when $\beta\ne -1$. When $\beta=-1$, it is not difficult to check that $1+xy=0$ is invariant and leads to the same result.
		On the other hand, every closed orbit contains a
		point with $y=0$, since $\dot x=y$.
		At this point, $xy>-1$.
		Consequently, every such orbit lies entirely in $U_0$.

		In the notation of~\eqref{k}, we have
		\[
		g(x)=x,\qquad P_1(x)=x^2,\qquad P_2(x)=0,
		\qquad P_3(x)=\beta.
		\]
		Choose $\varphi(x) = x/3$, which is the $C^2(\mathbb R)$
		extension of $P_1(x)/(3g(x))$. Then
		\[
		A_\varphi=0,\qquad
		B_\varphi=-\frac{x^3}{3},\qquad
		C_\varphi=\beta+\frac{1}{3}+\frac{2x^4}{27},
		\qquad \rho(x)=\exp\left(-\frac{x^4}{12}\right).
		\]
		System~\eqref{8} becomes
		\begin{equation}
			\begin{aligned}\label{13}
				\frac{d\xi}{d\sigma}&=\eta,\\
				\frac{d\eta}{d\sigma}
				&=-\xi \exp\left(-\frac{\xi^4}{6}\right)
				-\exp\left(\frac{\xi^4}{12}\right)
				\left(\beta+\frac{1}{3}+\frac{2\xi^4}{27}\right)\eta^3\\
				&=:-\widehat g(\xi)-\widehat f(\xi)\eta^3.
			\end{aligned}
		\end{equation}
		For $\beta<-1/3$, set
		\[
		c=\left(-\frac{27}{2}\left(\beta+\frac13\right)\right)^{1/4}.
		\]
		Then $C_\varphi<0$ for $|x|<c$ and $C_\varphi>0$
		for $|x|>c$. Moreover,   $\widehat g(\xi)$ is odd,  so $\widehat G_\varphi(\xi)$ is even. In particular,  $\widehat G_\varphi(-c)=\widehat G_\varphi(c)$. By  Theorem~\ref{thm3}, system \eqref{12} has at most one limit cycle in $U_\varphi=\{1+xy/3>0\}$, and
		any such cycle is hyperbolic and stable.
		Since $U_0\subset U_\varphi$, system \eqref{12} has at most one limit cycle in $\mathbb R^2$ when $\beta<-1/3$.
		
		For $\beta\geq-1/3$, the divergence of the vector field of system \eqref{13} satisfies 
		\[
		{\bf div}(\eta,-\widehat g(\xi)-\widehat f(\xi)\eta^3)=-3\widehat f(\xi)\eta^2=-3\exp\left(\frac{\xi^4}{12}\right)
		\left(\beta+\frac{1}{3}+\frac{2\xi^4}{27}\right)\eta^2\leq0
		\]
		and it is not identically zero on any nonempty open set.
		Thus,  by the Bendixson--Dulac theorem, system  \eqref{13} has no closed orbits  when   $\beta\geq-1/3$. The same conclusion holds
		for system \eqref{12} in $U_\varphi$. Since every closed orbit of system~\eqref{12}
		lies in $U_0\subset U_\varphi$, it follows that
		system~\eqref{12} has no closed orbits in $\mathbb R^2$.
		
		Finally, for $\beta\leq-1$,  apply transformations \eqref{4}--\eqref{8} with $\varphi^*(x)=x$.
		Then
		\[
		A_{\varphi^*}=-2x^2,\qquad
		B_{\varphi^*}=x^3,\qquad
		C_{\varphi^*}=\beta+1,\qquad
		\rho^*(x)=\exp\left(\frac{x^4}{4}\right).
		\]
		System~\eqref{8} becomes
		\[
		\begin{aligned}
			\frac{d\xi}{d\sigma}&=\eta,\\
			\frac{d\eta}{d\sigma}
			&=-\xi\exp\left(\frac{\xi^4}{2}\right)
			+2\xi^2\exp\left(\frac{\xi^4}{4}\right)\eta
			-(\beta+1)\exp\left(-\frac{\xi^4}{4}\right)\eta^3.
		\end{aligned}
		\]
		The divergence of its vector field is
		\[
		2\xi^2\exp\left(\frac{\xi^4}{4}\right)
		-3(\beta+1)\exp\left(-\frac{\xi^4}{4}\right)\eta^2
		\geq0,
		\]
		and is not identically zero on any nonempty open set.
		Thus, by the Bendixson--Dulac theorem, the transformed
		system has no closed orbits when $\beta\leq-1$. Thus, system \eqref{12} has no closed orbits in  $U_{\varphi^*}=U_0$.  
		Since every closed orbit of system~\eqref{12} lies in
		$U_0$, 
		system~\eqref{12} also has no closed orbits.
		This completes the proof.
	\end{proof}
	
	\subsection{A family of systems with homogeneous terms}
	
	Cima, Gasull and Ma\~nosas~\cite{CGM} constructed systems
	\[
	\frac{dx}{dt}=P_m(x,y),\qquad
	\frac{dy}{dt}=Q_n(x,y),
	\]
	with at least $(m+n)/2$ limit cycles, where $m\ne n$
	and $P_m$ and $Q_n$ are homogeneous polynomials
	of odd degrees $m$ and $n$, respectively.
	In 2021, Gasull~\cite[Problem~2]{Gasull} asked whether
	$(m+n)/2$ is also an upper bound for the number of
	limit cycles of such systems.
	This problem remains unresolved even in the simplest
	case $m=1$ and $n=3$;
	see~\cite[Problem~3(i)]{Gasull}.
	
	Gasull~\cite[Problem~3(ii)]{Gasull} also asked for
	a simple proof of the uniqueness of the limit cycle
	of the system
	\begin{equation}\label{14}
		\dot x=y,\qquad
		\dot y=-x^3+dx^2y+y^3,
		\qquad d\in\mathbb R,
	\end{equation}
	which is a subfamily of the above systems with
	$m=1$ and $n=3$.
	Using suitably constructed Dulac functions,
	Gasull and Giacomini~\cite{GG} established that
	system~\eqref{14} has at most one limit cycle for
	$-3/\sqrt[3]{2}<d<0$, and that any such cycle
	is hyperbolic and unstable.
	More recently, by analysing the heteroclinic
	separatrices at infinity, Zhuang and Liu~\cite{ZL}
	proved that there exists $d_*<0$ such that
	system~\eqref{14} has exactly one limit cycle for
	$d\in(d_*,0)$ and none otherwise.
	Here, we apply Theorem~\ref{thm2} to obtain an
	alternative proof of uniqueness and hyperbolicity
	for $-3/\sqrt[3]{2}\leq d<0$.
	
	\begin{proposition}\label{co2}
		If
		\[
		-\frac{3}{\sqrt[3]{2}}\leq d<0,
		\]
		system~\eqref{14} has at most one limit cycle
		in $\mathbb R^2$.
		If such a limit cycle exists, it is hyperbolic and unstable.
		If $d\geq0$, the system has no  closed  orbits.
	\end{proposition}
	
	\begin{proof}
		For $d\geq0$,  the divergence satisfies
		\[
		{\bf div}(y,-x^3+dx^2y+y^3)
		=dx^2+3y^2\geq0,
		\] and does not vanish identically on any nonempty open set. 
		Thus, by the Bendixson--Dulac theorem,
		system~\eqref{14} has no closed orbits.
		
		Next, suppose that $-3/\sqrt[3]{2}\leq d<0$. In the notation of~\eqref{k},
		we have
		\[
		g(x)=x^3,\qquad P_1(x)=-d x^2,\qquad
		P_2(x)=0,\qquad P_3(x)=-1.
		\]
		Consider the Cauchy initial value problem
		\[
		\varphi'
		=1+dx^2\varphi^2+x^3\varphi^3,
		\qquad \varphi(0)=0.
		\]
		Let $(-M,M)$, where $0<M\leq+\infty$, denote
		the maximal interval of existence of its solution. 
		By uniqueness of solutions, $\varphi(x)$ is odd. 
		Moreover, $\varphi'(0)=1$, and $\varphi'(x)=1$
		when $\varphi=0$. Hence $
		\varphi(x)>0$ for  $0<x<M$.

		We first claim that every closed orbit of
		system~\eqref{14} is contained in
		\[
		U_{\varphi,M}
		=\{(x,y):|x|<M,\ 1+\varphi(x)y>0\}.
		\]
		Setting $L(x,y)=1+\varphi(x)y$ for $|x|<M$,
		we obtain
		\[
		\dot L
		=L\left(
		-x^3\varphi+
		\bigl(x^3\varphi^2+d  x^2\varphi\bigr)y+y^2
		\right).
		\]
		Thus $\{L=0\}$ is invariant.
		Suppose first that $M<+\infty$.
		Consider $L(x,y)=1+\varphi(x)y$ for $x>0$.
		By maximality and the positivity of $\varphi$,
		\[
		\varphi(x)\longrightarrow+\infty
		\qquad\text{as }x\to M^-.
		\]
		Consequently, the invariant orbit
		\[
		y=-\frac{1}{\varphi(x)},\qquad 0<x<M,
		\]
		satisfies
		\[
		y\longrightarrow0\quad\text{as }x\to M^-,
		\qquad
		y\longrightarrow-\infty\quad\text{as }x\to0^+.
		\]
		Therefore, the invariant orbit can be extended
		to $(M,0)$ in finite backward time because $(M,0)$ is a regular point.
		In other words, the orbit of system~\eqref{14}
		through $(M,0)$ can be represented by $\{L=0\}$
		in the fourth quadrant.
		By existence and uniqueness of solutions,
		no closed orbit of system~\eqref{14} can intersect
		the invariant orbit
		\[
		y=-\frac{1}{\varphi(x)},\qquad 0<x<M,
		\]
		or pass through $(M,0)$.
		Moreover, since $\dot x=y$, the rightmost point
		of every closed orbit lies on the $x$-axis.
		If its $x$-coordinate were greater  than or equal to $M$,
		the lower arc of the closed orbit would necessarily
		intersect the invariant orbit above.
		Thus every closed orbit lies to the left of $x=M$.
		By symmetry, it also lies to the right of $x=-M$.
		Furthermore, every closed orbit 
		contains a point $(x_0,0)$ on the $x$-axis, where
		$L(x_0,0)=1>0$.
		Therefore, every closed orbit is contained in
		\[
		U_{\varphi,M}
		=\{(x,y):|x|<M,\ 1+\varphi(x)y>0\}.
		\]
		When $M=+\infty$, containment in $|x|<M$
		is immediate.
		Therefore, it suffices to consider the limit cycles
		of system~\eqref{14} in the strip $(-M,M)\times\mathbb R$.
		
		By construction, $C_\varphi\equiv0$, and
		\[
		A_\varphi(x)
		=x^2\bigl(-d-3x\varphi(x)\bigr),
		\qquad
		B_\varphi(x)
		=2d x^2\varphi(x)+3x^3\varphi^2(x).
		\]
		We next claim that there exists a unique $c\in(0,M)$
		such that
		\begin{equation}\label{11}
			A_\varphi(x)>0\quad(0<|x|<c),\qquad
			A_\varphi(x)<0\quad(c<|x|<M).
		\end{equation}
		Set $h(x)=-d-3x\varphi(x)$ for $0<x<M$.
		Then
		\[
		h'(x)=-3\varphi(x)
		-3x\bigl(1+dx^2\varphi^2(x)+x^3\varphi^3(x)\bigr).
		\]
		At any zero of $h(x)=0$, we have $x\varphi(x)=-d/3$,
		and hence
		\[
		h'(x)=\frac{d}{x}
		-3x\left(1+\frac{2d^3}{27}\right)<0,
		\]
		since $-3/\sqrt[3]{2}\leq d<0$.
		Thus $h$ has at most one zero in $(0,M)$. Moreover, if $M<+\infty$, then $h(0)=-d>0$ and
		$\lim_{x\to M^-}h(x)=-\infty$.
		Thus $h$ has at least one zero in $(0,M)$.
		If $M=+\infty$, suppose that $h$ has no zero. Then 
		$h(x)>0$ for $0<x<M$, and hence
		$0<x\varphi(x)<-d/3$ for all $x>0$.
		Consequently,
		\[
		\varphi'(x)
		=1+dx^2\varphi^2(x)+x^3\varphi^3(x)
		\geq1+\frac{2d^3}{27}\geq0.
		\]
		Thus,  $\varphi(x)$  is nondecreasing. Since  $\varphi(x)>0$  for $x > 0$, it follows that  $x\varphi (x)\to+\infty$  as $x\to +\infty$,
		contradicting $x\varphi(x)<-d/3$.
		Therefore $h$ also has at least one zero in $(0,M)$
		when $M=+\infty$.
		Consequently, there exists a unique $c\in(0,M)$
		such that
		\[
		h(x)>0\quad(0<x<c),\qquad
		h(x)<0\quad(c<x<M).
		\]
		Since $A_\varphi(x)=x^2h(x)$ and $\varphi$ is odd,
		we conclude \eqref{11}.

		Finally, we rewrite system~\eqref{8} as
		\[
		\begin{aligned}
			\frac{d\xi}{d\sigma}&=\eta,\\
			\frac{d\eta}{d\sigma}
			&=-\rho^2(\xi)\xi^3
			-\rho(\xi)A_\varphi(\xi)\eta,
		\end{aligned}
		\]
		where
		\[
		\rho(\xi)
		=\exp\left(\int_0^\xi \left(2d s^2\varphi(s)+3s^3\varphi^2(s)\right)\,ds\right).
		\]
		Since $\varphi(x)$ is odd, $\widehat g(\xi)=\rho^2(\xi)\xi^3$ is odd,
		and 
		$\widehat G_\varphi$ is even. In particular,
		\[
		\widehat G_\varphi(-c)=\widehat G_\varphi(c).
		\]
		Theorem~\ref{thm2} gives that system~\eqref{14}
		has at most one limit cycle in $U_{\varphi,M}$
		and any such cycle is hyperbolic and unstable.
		Since every closed orbit lies in $U_{\varphi,M}$,
		the conclusion holds in $\mathbb R^2$.
		This completes the proof.
	\end{proof}
	
	\subsection{A modified hybrid van der Pol--Rayleigh oscillator}
	Erlicher, Trovato and Argoul~\cite{ETA} introduced
	the modified hybrid van der Pol--Rayleigh oscillator
	\begin{equation}\label{w}
		\dot x=y,\qquad
		\dot y=-x-a_1y-(a_2x^2+a_3xy+a_4y^2)y,
	\end{equation}
	where $a_1<0$, $a_2\geq0$ and $a_4\geq0$,
	to model lateral pedestrian forces on a rigid floor.
	Chen et al.~\cite{CGLZ} studied system~\eqref{w}
	with $a_3>0$ and established the uniqueness of
	limit cycles.
	When $a_3<0$, the scaling
	$(x,y)\mapsto
	\bigl((-a_3)^{-1/2}x,(-a_3)^{-1/2}y\bigr)$
	transforms system~\eqref{w} into
	\begin{equation}\label{17}
		\dot x=y,\qquad
		\dot y=-x-ay-(\alpha x^2-xy+\beta y^2)y,
	\end{equation}
	where $a<0$, $\alpha\geq0$ and $\beta\geq0$.
	Lu and Zhang~\cite{LZ} investigated the bifurcation
	diagrams and global phase portraits of system~\eqref{17}.
	Although their Theorem~1.1 states that the system
	has at most one limit cycle throughout the parameter
	space, Remark~1.2 and Section~5 identify
	\[
	\{(\alpha,\beta,a):a<0,\ 0<\alpha\beta<1/4\}
	\]
	as an unresolved region for counting limit cycles.
	Their Conjecture~5.1 asserts that the system has
	exactly one limit cycle in this region, which is
	hyperbolic and stable.
	Proposition~\ref{co3}(ii) below provides an independent
	proof of uniqueness, hyperbolicity and stability
	under condition~\eqref{18}, provided that a limit
	cycle exists.
	However, Proposition~\ref{co3}(i) disproves the
	existence assertion over the full conjectured region.
	For example, $\alpha=\beta=1/4$ and $a=-2$
	belong to this region, but the corresponding system
	has no closed  orbits.
	\begin{proposition}\label{co3}
		Suppose that $\alpha,\beta>0$. Then the following statements hold.
		\begin{itemize}
			\item[(i)] If $\alpha+\beta\leq1$ and $a\leq-2$,
			then the system has no closed orbits.
			
			\item[(ii)] If
			\begin{equation}\label{18}
				\max\left\{-4\alpha-3\beta,-\frac{3}{\sqrt{2}}\right\}
				\leq a<0,
			\end{equation}
			then the system has at most one limit cycle in $\mathbb R^2$.
			If such a limit cycle exists, it is hyperbolic and stable.
		\end{itemize}
	\end{proposition}

	\begin{proof}
		We first prove the nonexistence assertion (i).
		Suppose that $\alpha+\beta\leq1$ and $a\leq-2$.
		On the line $y=x$ with $x>0$, we have
		\[
		\left.\frac{dy}{dx}\right|_{y=x}
		=-1-a+(1-\alpha-\beta)x^2\geq1.
		\]
		Since this line has slope $1$ and $\dot x=y>0$,
		orbits can only cross it from $y<x$ to $y>x$,
		unless $a=-2$ and $\alpha+\beta=1$, when it is invariant. 
		However, any closed
		orbit runs from the positive $y$-axis, where $y>x$,
		to its rightmost point $(b,0)$, where $y<x$.
		This contradiction excludes closed orbits.

		We now prove assertion (ii).
		In the notation of~\eqref{k}, we have
		\[
		g(x)=x,\qquad P_1(x)=a+\alpha x^2,\qquad
		P_2(x)=-x,\qquad P_3(x)=\beta.
		\]
		Consider the Cauchy initial value problem
		\begin{equation}\label{19}
			\varphi'
			=-\beta-x\varphi-(a+\alpha x^2)\varphi^2
			+x\varphi^3,\qquad \varphi(0)=0.
		\end{equation}
		Let $(-N,N)$, where $0<N\leq+\infty$, denote
		the maximal interval of existence of its solution.
		By uniqueness of solutions, $\varphi(x)$ is odd.
		Since $\varphi'(0)=-\beta< 0$ and $\varphi'(x)=-\beta$ whenever $\varphi(x)=0$, we have 	$\varphi(x)<0$ for $0<x<N$.
		We next show  that every closed orbit of
		system~\eqref{17} is contained in
		\[
		U_{\varphi,N}
		=\{(x,y):|x|<N,\ 1+\varphi(x)y>0\}.
		\]
		Indeed, set $L(x,y)=1+\varphi(x)y$.
		Since $C_\varphi\equiv0$, the set $\{L=0\}$ is invariant.
		If $N<+\infty$, then
		$\varphi(x)\to-\infty$ as $x\to N^-$.
		The invariant graph $y=-1/\varphi(x)$ in the first
		quadrant tends to $+\infty$ as $x\to0^+$ and to $0$
		as $x\to N^-$.
		A periodic orbit whose rightmost point has abscissa
		greater than $N$ would have its upper arc intersect
		this graph.
		Equality would place the periodic orbit on the
		nonperiodic orbit through $(N,0)$.
		Both alternatives contradict uniqueness of solutions.
		By central symmetry, every periodic orbit is contained
		in $|x|<N$.
		Since it contains a point with $y=0$, invariance of
		$\{L=0\}$ then gives $L>0$ along the orbit.
		The same conclusion holds immediately when $N=+\infty$. Therefore, it suffices to consider the limit cycles
		of system~\eqref{17} in the strip $(-N,N)\times\mathbb R$.
		
		Since $\varphi(x)$ satisfies \eqref{19}, we have 
		$C_\varphi\equiv0$ and
		\[
		\begin{aligned}
			A_\varphi(x)
			&=a+\alpha x^2-3x\varphi(x),\\
			B_\varphi(x)
			&=-x-2(a+\alpha x^2)\varphi(x)+3x\varphi^2(x).
		\end{aligned}
		\]
		We claim that there exists a unique $c\in(0,N)$
		such that
		\[
		A_\varphi(x)<0\quad(|x|<c),\qquad
		A_\varphi(x)>0\quad(c<|x|<N).
		\]
		Indeed, we consider $0<x<N$. By calculation,
		\[
		A_\varphi'(x)=2\alpha x-3\varphi(x)-3x\left(-\beta-x\varphi(x)-(a+\alpha x^2)\varphi^2(x)
		+x\varphi^3(x)\right).
		\]
		At any point where $A_\varphi(x)=0$ with $x>0$,
		we have
		\[
		3x\varphi(x)=a+\alpha x^2,\qquad
		a<a+\alpha x^2<0.
		\]
		Hence,
		\[
		xA_\varphi'(x)
		=-(a+\alpha x^2)
		\left(1-\frac{2(a+\alpha x^2)^2}{9}\right)
		+x^2(2\alpha+3\beta+a+\alpha x^2).
		\]
		By~\eqref{18}, $a^2\leq9/2$ and $a+4\alpha+3\beta\geq0$.
		At any positive zero of $A_\varphi$,  since
		$a<a+\alpha x^2<0$, we have
		\[
		0<\frac{a+\alpha x^2}{a}<1.
		\]
		If $a+4\alpha\geq0$, then
		\[
		3\beta>0\geq
		-\frac{(a+4\alpha)(a+\alpha x^2)}{a}.
		\]
		If $a+4\alpha<0$, then
		\[
		3\beta\geq-(a+4\alpha)\geq
		-\frac{(a+4\alpha)(a+\alpha x^2)}{a}.
		\]
		Consequently,
		\[
		\begin{aligned}
			xA_\varphi'(x)
			&\geq
			-(a+\alpha x^2)
			\left(1-\frac{(a+\alpha x^2)^2}{a^2}\right)
			+x^2\left(2\alpha
			-\frac{4\alpha(a+\alpha x^2)}{a}\right)\\
			&=\frac{\alpha^2x^4(\alpha x^2-a)}{a^2}>0.
		\end{aligned}
		\]
		Thus $A_\varphi$ has at most one zero in $(0,N)$. If $N<+\infty$, then $\varphi(x)\to -\infty$ as $x\to N^-$, so $A_\varphi(x)\to+\infty$. Since $A_\varphi(0)=a<0$,
		continuity implies that $A_\varphi(x)$ has at least one zero in $(0,N)$.
		When $N=+\infty$, since $\varphi(x)<0$ for $0<x<N$,
		\[
		A_\varphi (x)\geq a+\alpha x^2,
		\]
		so 
		$\lim_{x\to+\infty}A_\varphi (x)=+\infty$.
		Thus $	A_\varphi (x)$ has at least one zero in $(0,N)$.
		Consequently, there exists a unique $c\in(0,N)$
		such that
		\[
		A_\varphi(x)<0\quad(0<x<c),\qquad
		A_\varphi(x)>0\quad(c<x<N).
		\]
		Since  $\varphi$ is odd,
		the claimed signs of $A_\varphi$ follow.

		Finally, let
		\[
		\rho(x)
		=\exp\left(\int_0^x B_\varphi(s)\,ds\right).
		\]
		System~\eqref{8} becomes
		\[
		\begin{aligned}
			\frac{d\xi}{d\sigma}&=\eta,\\
			\frac{d\eta}{d\sigma}
			&=-\rho^2(\xi)\xi
			-\rho(\xi)A_\varphi(\xi)\eta.
		\end{aligned}
		\]
		Since $\varphi$ is odd, $B_\varphi$ is odd and
		$\rho$ is even.
		Hence $\widehat g(\xi)=\rho^2(\xi)\xi$ is odd,
		and 
		$\widehat G_\varphi$ is even. In particular,
		\[
		\widehat G_\varphi(-c)=\widehat G_\varphi(c).
		\]
		Theorem~\ref{thm2} therefore gives that
		system~\eqref{17} has at most one limit cycle
		in $U_{\varphi,N}$, and any such cycle is hyperbolic
		and stable. Since every closed orbit lies in $	U_{\varphi,N}$,
		the conclusion holds in $\mathbb R^2$.
		This completes the proof.
	\end{proof}
	
	\subsection{Comparison with existing uniqueness criteria} 
	In this subsection, we recall a uniqueness criterion
	due to Chen, Tang and Xiao~\cite{CTX2023} and explain
	why it does not apply directly to our three examples
	in the parameter regions of interest.
	
	\begin{theorem}[{Chen, Tang and Xiao,
			\cite[Theorem~1]{CTX2023}}]\label{t11}
		Consider
		\begin{equation}\label{1q}
			\dot x=y,\qquad
			\dot y=-|x|^m\operatorname{sgn}(x)-f(x,y)y,
		\end{equation}
		where $1\leq m<+\infty$ and $f$ is continuous
		on $\mathbb R^2$, with sufficient regularity
		to ensure existence and uniqueness of solutions
		to the initial value problems.
		Suppose that, for all sufficiently small
		$\varepsilon>0$ and every $0<k<1$,
		\begin{equation}\label{111}
			\begin{aligned}
				k^{(1-m)/2}
				f\bigl(kx,k^{(1+m)/2}y\bigr)
				&\leq f(x,y),
				&&\forall(x,y)\in\mathbb R^2,\\
				k^{(1-m)/2}
				f\bigl(kx,k^{(1+m)/2}y\bigr)
				&\ne f(x,y),
				&&x\ne0,\ 0<|y|<\varepsilon.
			\end{aligned}
		\end{equation}
		Then system~\eqref{1q} has at most one limit cycle.
		If such a cycle exists, it is stable.
	\end{theorem}
	
	For $0<k<1$, set
	\[
	\mathscr D_{k,m}(x,y)
	=k^{(1-m)/2}f\bigl(kx,k^{(m+1)/2}y\bigr)-f(x,y).
	\]
	For the three polynomial families considered
	in this paper, direct calculation gives
	\[
	\begin{aligned}
		\mathscr D_{k,1}(x,y)
		&=(k^2-1)(x^2+\beta y^2),
		&&\text{for system~\eqref{12}},\\
		\mathscr D_{k,3}(x,y)
		&=d(1-k)x^2+(1-k^3)y^2,
		&&\text{for system~\eqref{14}},\\
		\mathscr D_{k,1}(x,y)
		&=(k^2-1)(\alpha x^2-xy+\beta y^2),
		&&\text{for system~\eqref{17}}.
	\end{aligned}
	\]
	For every $0<k<1$, the first expression changes
	sign when $-1<\beta<-1/3$, and the second changes
	sign when $d<0$.
	For the third family, the quadratic form
	$\alpha x^2-xy+\beta y^2$ is indefinite when
	$0<\alpha\beta<1/4$, so the corresponding expression
	also changes sign.
	Thus condition~\eqref{111} fails,
	and Theorem~\ref{t11} does not apply directly
	in these parameter regions.
	
	{\bf Conflict of
		interest:} The authors have no conflicts of interest to declare that are relevant to the content of this article.
	
	{\bf Data Availability:}   No datasets were generated or analysed in this study.

	\section*{Acknowledgements}
	The work   is  supported by  Hunan Basic Science Research Center for Mathematical Analysis (2024JC2002). The first and the third authors are   supported by 
	the  National Natural Science Foundation of China (Nos. 12322109, 12571190). 	The second author is supported by the National Key R \& D Program of China (No. 2022YFA1005900), the Innovation Program of Shanghai Municipal Education Commission (No. 2021-01-07-00-02-E00087), and the National Natural Science Foundation of China (No. 12271353). The third author is supported by the Postdoctoral Fellowship Program and China Postdoctoral Science Foundation under Grant Number BX2026012.


\begin{thebibliography}{99}
		
		\bibitem{CGLZ}
		H.~Chen, K.~Gebreselassie, Y.~Lu, R.~Zhang,
		Global dynamics of a modified hybrid Van der Pol--Rayleigh oscillator,
		\textit{Acta Math. Sin. (Engl. Ser.)}, to appear.
		
		\bibitem{CTX}
		H.~Chen, Y.~Tang, D.~Xiao,
		Global dynamics of hybrid van der Pol--Rayleigh oscillators,
		\textit{Phys. D}, \textbf{428} (2021), 133021.
		
		\bibitem{CTX2023}
		H.~Chen, Y.~Tang, D.~Xiao,
		On the uniqueness of limit cycles in second-order oscillators,
		\textit{J. Differential Equations}, \textbf{370} (2023), 140--166.
		
		\bibitem{CX}
		H.~Chen, D.~Xiao,
		On the existence and uniqueness of limit cycles for hybrid oscillators,
		\textit{Ann. Mat. Pura Appl.}, \textbf{202} (2023), 2049--2071.
		
		\bibitem{Cherkas1978}
		L.~A.~Cherkas,
		Conditions for a center for the equation
		$yy'=\sum_{i=0}^{3}p_i(x)y^i$,
		\textit{Differ. Uravn.}, \textbf{14} (1978), 1594--1600 (in Russian);
		English translation: \textit{Differ. Equ.}, \textbf{14} (1978), 1133--1137.
		
		\bibitem{CLT}
		C.~Christopher, C.~Li, J.~Torregrosa,
		\textit{Limit Cycles of Differential Equations},
		2nd ed., Birkh\"auser, Cham, 2024.
		
		\bibitem{CGM}
		A.~Cima, A.~Gasull, F.~Ma\~nosas,
		Limit cycles for vector fields with homogeneous components,
		\textit{Appl. Math. (Warsaw)}, \textbf{24} (1997), 281--287.
		
		\bibitem{Dulac1923}
		H.~Dulac,
		Sur les cycles limites,
		\textit{Bull. Soc. Math. France}, \textbf{51} (1923), 45--188.
		
		\bibitem{Ecalle1992}
		J.~\'Ecalle,
		\textit{Introduction aux fonctions analysables et preuve constructive
			de la conjecture de Dulac},
		Hermann, Paris, 1992.
		
		\bibitem{ETA}
		S.~Erlicher, A.~Trovato, P.~Argoul,
		Modeling the lateral pedestrian force on a rigid floor
		by a self-sustained oscillator,
		\textit{Mech. Syst. Signal Process.}, \textbf{24} (2010), 1579--1604.
		
		\bibitem{FD}
		A.~C.~de Pina Filho, M.~S.~Dutra,
		Application of hybrid van der Pol--Rayleigh oscillators
		for modeling of a bipedal robot,
		\textit{Mech. Solid Braz.}, \textbf{1} (2009), 209--221.
		
		\bibitem{Gasull}
		A.~Gasull,
		Some open problems in low dimensional dynamical systems,
		\textit{SeMA J.}, \textbf{78} (2021), 233--269.
		
		\bibitem{GG}
		A.~Gasull, H.~Giacomini,
		Upper bounds for the number of limit cycles
		through linear differential equations,
		\textit{Pacific J. Math.}, \textbf{226} (2006), 277--296.
		
		\bibitem{Hale}
		J.~K.~Hale,
		\textit{Ordinary Differential Equations},
		2nd ed., Robert E. Krieger Publishing Co., Huntington, NY, 1980.
		
		\bibitem{Ilyashenko1990}
		Yu.~Ilyashenko,
		Finiteness theorems for limit cycles,
		\textit{Russian Math. Surveys}, \textbf{45} (1990), no.~2, 129--203;
		Russian original: \textit{Uspekhi Mat. Nauk},
		\textbf{45} (1990), no.~2, 143--200.
		
		\bibitem{Ilyashenko1991}
		Yu.~Ilyashenko,	  \textit{Finiteness Theorems for Limit Cycles}, Translations of Mathematical Monographs, vol. 94, American Mathematical Society, Providence, RI, 1991.
		
		\bibitem{kukles}
		I.~S.~Kukles,
		Sur quelques cas de distinction entre un foyer et un centre,
		\textit{C. R. (Doklady) Acad. Sci. URSS (N.S.)},
		\textbf{42} (1944), 208--211.
		
		\bibitem{LS}
		N.~Levinson, O.~K.~Smith,
		A general equation for relaxation oscillations,
		\textit{Duke Math. J.}, \textbf{9} (1942), 382--403.
		
		\bibitem{LC}
		L.-P.~Liou, K.-S.~Cheng,
		On the uniqueness of a limit cycle for a predator-prey system,
		\textit{SIAM J. Math. Anal.}, \textbf{19} (1988), 867--878.
		
		\bibitem{LZ}
		Y.~Lu, R.~Zhang,
		Global dynamics of a modified hybrid Van der Pol--Rayleigh oscillator II,
		\textit{Electron. J. Differential Equations},
		\textbf{2026} (2026), no.~25, 1--17.
		
		\bibitem{NB}
		A.~H.~Nayfeh, B.~Balachandran,
		\textit{Applied Nonlinear Dynamics:
			Analytical, Computational, and Experimental Methods},
		Wiley-VCH, Weinheim, 2004.
		
		\bibitem{OAI2026}
		OpenAI,
		Uniform bounds for planar polynomial limit cycles,
		preprint, 2026,
		\href{https://github.com/openai/math/blob/main/preprints/uniform-bounds-for-planar-polynomial-limit-cycles-September-24-2026/uniform-bounds-for-planar-polynomial-limit-cycles-September-24-2026.pdf}{available online}.
		
		\bibitem{Sansone1949}
		G.~Sansone,
		Sopra l'equazione di A.~Li\'enard delle oscillazioni di rilassamento,
		\textit{Ann. Mat. Pura Appl.}, \textbf{28} (1949), 153--181.
		
		\bibitem{WK}
		X.~Wang, R.~E.~Kooij,
		Limit cycles in a cubic system with a cusp,
		\textit{SIAM J. Math. Anal.}, \textbf{23} (1992), 1609--1622.
		
		\bibitem{XZ}
		D.~Xiao, Z.~Zhang,
		On the uniqueness and nonexistence of limit cycles
		for predator-prey systems,
		\textit{Nonlinearity}, \textbf{16} (2003), 1185--1201.
		
		\bibitem{XZ2008}
		D.~Xiao, Z.~Zhang,
		On the existence and uniqueness of limit cycles
		for generalized Li\'enard systems,
		\textit{J. Math. Anal. Appl.}, \textbf{343} (2008), 299--309.
		
		\bibitem{Yeung}
		M.~Yeung,
		Dulac's theorem revisited,
		\textit{Qual. Theory Dyn. Syst.}, \textbf{24} (2025), Paper No.~57.
		
		\bibitem{Zeng}
		X.~Zeng,
		On the uniqueness of limit cycle of Li\'enard's equation,
		\textit{Sci. Sinica Ser. A}, \textbf{25} (1982), no.~6, 583--592.
		
		\bibitem{ZZG}
		X.~Zeng, Z.~Zhang, S.~Gao,
		On the uniqueness of the limit cycle
		of the generalized Li\'enard equation,
		\textit{Bull. London Math. Soc.}, \textbf{26} (1994), 213--247.
		
		\bibitem{Zh}
		Z.~Zhang, T.~Ding, W.~Huang, Z.~Dong,
		\textit{Qualitative Theory of Differential Equations},
		Translations of Mathematical Monographs, vol.~101,
		American Mathematical Society, Providence, RI, 1992.
		
		\bibitem{ZL}
		Z.~Zhuang, C.~Liu,
		A complement to the uniqueness of the limit cycle
		of a family of systems with homogeneous components,
		\textit{Qual. Theory Dyn. Syst.}, \textbf{23} (2024), Paper No.~205.
		
	\end{thebibliography}
\end{document}